\documentclass[12pt]{amsart}

\usepackage{amsmath}
\usepackage{amssymb, amsthm}
\usepackage[pdftex]{graphicx}
\usepackage{setspace}
\usepackage{verbatim}
\usepackage{tikz-cd}
\usetikzlibrary{arrows, positioning}
\usepackage{amscd}

\usepackage{epsf}
\usepackage{latexsym}
\usepackage{caption}
\input epsf.tex

\newcommand{\Z}{\mathbf Z}

\newcommand{\Zz}{\mathbf {Z}}

\theoremstyle{plain}
\newtheorem{theorem}{Theorem}[section]
\newtheorem{corollary}[theorem]{Corollary}
\newtheorem{lemma}[theorem]{Lemma}

\newtheorem{proposition}[theorem]{Proposition}

\theoremstyle{definition}
\newtheorem{definition}[theorem]{Definition}
\newtheorem{example}[theorem]{Example}

\newtheorem{fact}[theorem]{Fact}

\numberwithin{equation}{section}

\title[Extensions of Positive type Transformations ]{On $v$-positive type Transformations in Infinite Measure}

\author[P\u{a}durariu]{Tudor P\u{a}durariu}
\address[Tudor P\u{a}durariu]{University of California, Los Angeles, CA 90095-1555, USA }
\email {tudor\_pad@yahoo.com}

\author[Silva]{Cesar E. Silva}
\address[Cesar Silva]{Department of Mathematics\\
     Williams College \\ Williamstown, MA 01267, USA}
\email{csilva@williams.edu}

\author[Zachos]{Evangelie Zachos}
\address[Evangelie Zachos]{Princeton University, Princeton, NJ 08544, USA}
\email{evangelie.zachos@gmail.com}


\begin{document}
\bibliographystyle{alpha} 

\maketitle

 \begin{abstract}
For each vector $v$ we define the notion of a $v$-positive type for infinite measure-preserving transformations, a refinement of positive type as introduced by Hajian and Kakutani. We prove that a positive type transformation need not be  $(1,2)$-positive type.
We study this notion in the context of Markov shifts and multiple recurrence and give several examples.
\end{abstract}

\section{Introduction}

Let $(X, \mathcal B, \mu)$ be a standard Borel measure space with a sigma-finite nonatomic measure. Let  $T:X\to X$ be a measure-preserving transformation (i.e., $T^{-1}(A)$ is measurable and  $\mu(T^{-1}A)=\mu(A)$ for all $A\in\mathcal B$). We will assume that $T$ is invertible and $T^{-1}$ is measure preserving, although the definitions given can be easily generalized to the noninvertible case. We concentrate on measure spaces with infinite measure. A transformation is ergodic if every invariant set is null or conull; it is conservative it if admits no wandering sets of positive measure. It follows that a transformation $T$ is conservative ergodic if for every pair of measurable sets $A$ and $B$ there exists a positive integer $n$ such that $\mu(A\cap T^{-n}B)>0$. If an invertible measure-preserving transformation is ergodic on  a nonatomic space, then it is conservative. All of our transformations will be assumed to be conservative ergodic. 

To classify  infinite measure-preserving transformations, Hajian and Kakutani introduced the concepts of zero and positive type in \cite{HK64}.
A transformation $T$ is of  \emph{zero type} if for all sets $A$ of finite measure 
\[\lim_{n\to \infty} \mu(T^{-n}A\cap A) = 0.\]
A transformation  $T$ is of  \emph{positive type} if for all sets $A$ of positive, finite measure 
\[\limsup_{n\to \infty} \mu(A\cap T^{-n}A) > 0.\]

 If $T$ is infinite measure-measure preserving and ergodic, the Birkhoff ergodic theorem immediately implies that
  \[\liminf_n \mu(T^{-n}A\cap A)=0\,\,\text {for all finite measure } A.\] 
The $\limsup$ definition for positive type thus captures the greatest possible variation in the limit behavior of these intersections.
In \cite{HK64}, Hajian and Kakutani note that conservative ergodic infinite measure-preserving transformations are either zero type or positive type. Additionally, it is known that such a transformation $T$ is zero type if and only if all powers $T^k$ are zero type (where the proof does {\emph not} depend on ergodicity).

 In this paper, we introduce a new concept inspired by the definition of positive type, called $v$-positive type, for each vector $v$ in $\Z^d$. 
This notion provides a further classification of positive type transformations. We also define a concept related to  $v$-positive type that we call  $v$-positive multiplicative type.
In Section~\ref{S:vpositive} we prove that $v$-positive multiplicative type is  equivalent to   the corresponding Cartesian product transformation being positive type. It is a direct consequence of the definitions that $v$-positive type implies $v$-positive multiplicative type, but we prove that 
the converse does not hold. In fact, in Section~\ref{S:PTexamples} we prove that for every $v$ such that $v_d/v_1\geq d$, there is a transformation that is $v$-multiplicative positive type 
but not $v$-positive type. We also construct $v$-positive type transformations for each $v$. These constructions are rank-one examples and use approximation results from Section~\ref{S:approx}. 
In Section~\ref{S:Markov}, we use results involving Markov shifts to discuss 
connections with multiple recurrence. Section 7 discusses varying the vector $v$ to compare $v$-positive type and $w$-positive type.

\subsection{Acknowledgements}
This paper arose out of research from the ergodic theory group of the 2012 SMALL Undergraduate Research Project at Williams College.
Support for this project was provided by the National Science Foundation REU Grant DMS - 0353634 and the Bronfman Science Center of Williams College. We reserve special thanks to the other participants in the 2012 ergodic theory group: Irving Dai, Xavier Garcia, Shelby Heinecke, and Emily Wickstrom.  We would also like to thank Stanley Eigen and Arshag Haijian for some discussions during the 2012 Williams Ergodic Theory Conference. Finally, we would like to thank Emmanuel Roy for comments on Section 5.

\section{$v$-positive type and $v$-multiplicative-positive type}\label{S:vpositive}

We begin with a vector, $v=(v_1,v_2,\ldots, v_d)\in (\mathbb{Z}^+)^d$. We study each vector essentially independently of all other vectors until section 7. Section 6 provides some motivation for the following definitions. 
\definition{A transformation $T$ is $v$-positive type for a set $A$ if $$\limsup_{n\to\infty}\mu(A\cap T^{v_1 n}A \cap\ldots\cap T^{v_d n}A)>0.$$ We define $T$ to be $v$-positive type if $T$ is $v$-positive for all sets of positive finite measure.}

\definition{A transformation $T$ is $v$-multiplicative-positive type for a set $A$ if $$\limsup_{n\to\infty}\mu(A\cap T^{v_1 n}A)\ldots\mu(A\cap T^{v_d n}A)>0$$ We similarly define $T$ to be $v$-multiplicative-positive type if $T$ is $v$-multiplicatively-positive type for all sets of positive finite measure. }\\

Of course, these definitions can only be satisfied if $T$ is positive type. They offer, however,  a way  for further classifying into subcategories different positive type transformations. 
Note that we can easily confine ourselves to vectors in $(\mathbf{Z}^+)^d$, for given $(v_1, \ldots, v_d)\in \mathbf{Z}^d$ where $v_1$ is the negative number of largest absolute value, then $0<\mu(A\cap T^{v_1n}A\cap \ldots \cap T^{v_dn} A)=\mu(T^{-v_1n}A\cap A \cap T^{(-v_1+v_2)n} A\cap \ldots \cap T^{(-v_1+v_d)n}A)$, so that we can investigate the $v$-positive type of $T$ by evaluating the $w$-positive type of $T$ where $w=(v_2-v_1, v_3-v_1, \ldots, v_d-v_1, -v_1)\in (\mathbf{Z}^+)^d$. It is typical to list these vectors with $v_i$ distinct and $v_{i+1}> v_i$.

In general, of the two definitions, $T$ $v$-multiplicative-positive type is easier to work with. The following theorem explains some of this ease, where the proof technique is adapted from Aaronson and Nakada \cite{AN00}. We start with two lemmas before presenting the theorem.

\begin{lemma}  \label{savior}
Suppose that $T$ is conservative ergodic and $v$-multiplicative-positive type. Then, for every collection $A_1, \ldots, A_d$ of sets of positive, finite measure, $$\limsup\mu(A_1\cap T^{v_1n} A_1) \ldots\mu (A_d \cap T^{v_d n}A_d)>0$$
\end{lemma}

\begin{proof}
As T is ergodic, we can find $c_2, \ldots, c_d$ such that $$C=A_1\cap T^{c_2} A_2 \cap \ldots \cap T^{c_d}A_d$$ has $\mu(C)>0$. Then, $\limsup \mu(C\cap T^{v_1n} C) \ldots  \mu(C\cap T^{v_dn}C)>0$ but we note that $\mu(A_j \cap T^{v_jn} A_j) = \mu(T^{c_j} A_j\cap T^{v_jn+c_j} A_j) \ge \mu(C\cap T^{nv_j} C)$ so that we can deduce the desired result. 
\end{proof}

As a corollary which is not used in the main proof of the following theorem, we have: 
\begin{corollary}Suppose that $T$ is conservative ergodic and is positive type. Then, for every finite collection of sets $A_1,..., A_m$ of positive, finite measure, we have 
\[\limsup \mu(A_1\cap T^{n}A_1)\ldots \mu(A_m\cap T^{n}A_m) >0. \]
\end{corollary}

The first part of the following lemma is standard.

\begin{lemma} Let $T$ be conservative ergodic and measure preserving transformation and let $k\in\mathbb{Z^+}$. Then there exists a measurable set $B$ and an integer $1\le \ell \le k$ such that $B, TB, \ldots, T^{\ell-1}B$ is a partition of $X$ and every positive measure invariant subset of $T^k$ contains ($\mod\mu$) at least one of $B$, $TB$,\,\ldots, $T^{\ell-1}B$.\label{pain}\end{lemma}

\begin{proof}

First, consider all $T^k$-invariant sets of positive measure. Given  such a set $A$, we can create a partition $\mathcal{B}$ with base set $B$ from $A$ in the following canonical method: we consider intersections of the form $A\cap T^{i_1} A\cap \ldots \cap T^{i_{r}}A$ where $i_j < k$ and the $i_j$ are distinct. There are a finite number of such sets, and we pick $B=A\cap T^{i_1} A\cap \ldots \cap T^{i_r}A$ so that $B$ is a set with positive measure with highest possible value of $r$ and then, to break any ties, minimal measure. (When all powers of $T$ are conservative ergodic, for example, both $A$ and $B$ will be forced to be $X$, the entire space). Now, as $A$ is $T^k$-invariant, by definition $B$ is as well, and so $B\cup \ldots \cup T^{k-1}B=X$ by the ergodicity of $T$. We then define $\ell=\ell_B$ to be the smallest positive number such that $B\cap T^\ell B =B$. Such an $\ell\le k$ must exist by the $T^k$-invariance of $B$. (For example, note that if $T^k$ is ergodic, then $\ell=1$. In fact, one can construct examples $B,T$ for any $\ell|k$.) Then by the maximality of $r$,  if $1\le j<\ell$, then $\mu(B\cap T^j B)=0$. In fact, if $1\le j< i< \ell$, then $\mu(T^i B\cap T^j B)=\mu(T^{i-j}B \cap B)=0$ by the same argument.  By our choice of $\ell$, $B\cup \ldots \cup T^{\ell-1}B=X$. Thus, $\mathcal{B}=\{ B, TB, \ldots, T^{\ell-1}B\}$ is  a partition of $X$. 

Now, we use a minimizing argument. Given any two partition bases, $B$ and $C$, there is some choice of $i$ such that $\mu(B\cap T^i C)>0$. Let $\mathcal{S}$ be the set of 
partition bases. Note of course that every set in $\mathcal{S}$ has positive measure and $\mathcal{S}$ is nonempty. Decreasing chains in $\mathcal{S}$ are of the form $\ldots \subset B\subset C \subset D$ with positive measure inclusion. Next, note that for each $B\in \mathcal{S}$, $B\cup T^1 B\cup \ldots \cup T^{\ell_B} B=X$. Now, if $B\subset C$, note that $\mu(X\backslash(B\cup T^1 B\cup \ldots \cup T^{\ell_C} B))>0$ as $\mu(C\backslash B)>0$ and $C$ is a partition base, so we can conclude that $\ell_B>\ell_C$. As $\ell_B\le k$ for all $B\in \mathcal{S}$, we see that chains in $\mathcal{S}$ must be finite. In fact, as we have a bound $k$, there must be a maximum $\ell_B$ for $B\in \mathcal{S}$ and we pick a $B$ with corresponding maximum $\ell_B$, noting that it necessarily must be the base set of a minimal partition, i.e. a base not positive-measure containing the base set to any other partition.  

Given a $T^k$-invariant set $A$ of positive measure, chosen without loss of generality such that $\mu(B\cap A)>0$, we wish to show that $B\subset A \mod \mu$. If, not, i.e. if $\mu(B\backslash A)>0$, then 
$C=A\cap B$ is $T^k$ invariant as well and $C\subset B$ with positive measure inclusion and so we can find the corresponding partition base set $C'\subset C$ which has the property that $C'\subset B$ with strict $\mu$-inclusion. This violates the minimality property of $B$, and so we can conclude that $B\subset A \mod\mu$. \end{proof}

\begin{theorem}
Let $T$ be a conservative ergodic  transformation on the space $(X, \mathcal{B}, \mu)$ and let $v=(v_1,...,v_d)$ where each $v_i\in \mathbb{Z}^+$.  Then $T^{v_1}\times\ldots\times T^{v_n}$ is positive type if and only if $T$ is $v$-multiplicative-positive type. 
\end{theorem}

\begin{proof}
One direction is clear. If $T^{v_1}\times\ldots\times T^{v_n}$ is positive type, then in particular it is positive type for the set $A\times\ldots\times A$ for each $A\subset X$ of positive finite measure, so $T$ clearly is $v$-multiplicative-positive type with respect to each of these $A$.\\ \\
 For the other direction, starting from the assumption that $T$ has $v$-multiplicative-positive type, we proceed as follows.
Let $S:=T^{v_1}\times...\times T^{v_d}$ be a measure preserving transformation on the space $(X^d, \mathcal{B}^d, \mu^d= \mu \times \ldots \times \mu).$ To show that $S$ has positive type, define the set $$\mathcal{W}:= \left\lbrace 
B \in \mathcal{B}^d : 0<\mu^d(B)<\infty, \lim\mu^d(B\cap S^nB)=0\right\rbrace.$$
Then, by a standard argument, we can find a set $Z$ which is a countable union of sets in $\mathcal{W}$ such that every set in $\mathcal{W}$ is contained up to a set of measure zero in $Z$. We observe that $$\lim \mu^d(B\cap S^nC)=0 $$ for all finite measure sets $B,C \in \mathcal{W}$. This can be seen in a short number of steps:
If $\limsup \mu^d(B\cap S^n C)>0$, we could find a sequence $\{n_i\}$ satisfying $\mu^d(B\cap S^{n_i} C)> \epsilon$ for a fixed $\epsilon>0$ and then as $B$ and $C$ are each of finite measure, we could find $N\in n_i$ and a subsequence $\{n_j'\}$ such that $\mu^d(B\cap S^N C \cap S^{n'_j} C)>\epsilon'$ for some fixed $\epsilon'>0$. Of course, then we have that $\limsup \mu^d(S^N C \cap S^{n'_j} C)\ge \epsilon'$, which is a contradiction of the defining property of $\mathcal{W}$.

This fact gives us that if $A, B\in \mathcal{W}, $then $A\cup B\subset \mathcal{W}$. If $A\subset Z$ has $\mu(A)<\infty$, then it can be approximated  up to epsilon measure by sets in $\mathcal{W}$, which then quickly implies that $A\in \mathcal{W}$. So, the argument above gives that $\lim \mu^d(B\cap S^nC)=0 $ for $B, C\subset Z$ where $\mu(B), \mu(C)<\infty$. \\

Returning to the main proof, it suffices to show that if $T$ is $v$-multiplicative-positive type, $\mu(Z)=0$. We will stop briefly to note that if $T^{v_i}$ is ergodic for some $1\le i\le d$, then the argument in \cite[Proposition 2.2 ]{AN00} can be used to finish our proof. However, the following proof is more general, working even when $T^{v_i}$ is non-ergodic for all $i$. To continue, assume for the sake of contradiction that $Z$ has positive measure in $X^d$.  \\ 

If we define $R_i:= I\times...\times T^{v_i}\times...\times I$, then for any $A\in \mathcal{W}$,$$\mu^d (R_i(A) \cap SR^i(A))=\mu^d(A\cap S)=0$$ so that $R_iA\in \mathcal{W}$ and therefore $R_i Z=Z\mod \mu$.

 Then define  $p_i: Z \to X, $ and $\pi_i: Z \to X^{d-1}$ with $p_i((x_1, \ldots, x_d) = x_i$ and $\pi_i((x_1, \ldots, x_d))= (x_1, \ldots, x_{i-1}, x_{i+1}, \ldots, x_d)$. We start an inductive process, beginning with $i=1$. First note that for any measurable subset $A$, by our restrictions at the beginning of the paper, we can conclude that $p_i(A)$ is measurable. For $x\in \pi_1(Z)$, we define $V^1_x:= p_1(\pi_1^{-1}(x))$, a measurable set. $V_x^1$ is a $T^{v_1}$-invariant subset of $X$, and so if it is of positive measure, applying Lemma \ref{pain} to $k=v_1$, it contains one $T^{j(x)}B_1$of $B_1, T^1 B_1, \ldots, T^{\ell-1}B$. If $D_1\subset \pi_1(Z)$ is the subset containing $x$ such that $\mu(V_x^1)>0$, then as $Z$ has positive measure, $\mu^{d-1}(D_1)>0$. We then have a map $f: D_1 \to \{1, \ldots, \ell\}$. There must be some $j$ such that $f^{-1}(j)$ has positive measure. In fact, define $C_1=\{ x\in D_1\text{ such that } T^jB\subset V^1_x\}$, where $f^{-1}(j)\subset C_1$, and we then define $A_1=T^{j}B_1$ and note that $Z_1=A_1\times C_1\subset Z$ has positive measure. Furthermore, for $i\ge 2$, we see that $R_i(Z_1)=Z_1$, and so we can proceed inductively to find $Z_d \subset Z_{d-1} \subset \ldots \subset Z_1 \subset Z$ where $Z_d$ has positive measure and in fact $A_1\times \ldots \times A_d \subset Z_d\subset Z$. For each $i$, take a subset of positive, finite measure $A'_i\subset A_i$.\\
\indent
By the Lemma \ref{savior} and the fact that $T$ is $v$-multiplicative-positive type, we can conclude that $\limsup \mu^d(A'_1\times \ldots \times A'_d \cap S(A'_1\times \ldots \times A'_d))>0$ so that $\Pi_i A'_i\not\subseteq Z$. Thus, by contradiction, $Z$ is trivial, and so $S$ is positive type.

\end{proof}

\noindent Therefore, $v$-multiplicative-positive type of $T$ is just another name for positive type of $T^{v_1}\times\ldots \times T^{v_d}$. All results obtained for positive type transformations apply. For example, if $T^{v_1} \times \ldots \times T^{v_d}$ is ergodic, a transformation is either $v$-multiplicative-positive type or $v$-multiplicative-zero type. For this reason, the remainder of the paper will focus primarily on $v$-positive type, a property that is less well understood.

\section{Rank-One Constructions and Approximation}\label{S:approx}

Rank-one transformations are a much-studied type of transformation, easily described by defining the transformation iteratively on intervals. We begin with a review of the cutting and stacking method of constructing rank-one examples and then proceed to a number of lemmas which that assist the move from considering columns and levels to considering arbitrary positive measure sets, as far as it can be accomplished.

A column, denoted here by $C$, is an ordered set of $h>0$ pairwise disjoint intervals in $\mathbb{R}$ of the same measure. (Cutting and stacking can be adapted to construct nonsingular examples, but we do not consider them here.) These intervals are called levels, and $C$ is said to have height $h$. We consider these levels to be ``stacked'' so that the element $i+1$ is directly above $i$, for all $0\leq i \leq h-2$. For a level $I$ in a column $C$, denote by $h(I)\in\{0,\ldots,h-1\}$ the height of $I$, its position in $C$. Let $w(C)$ denote the measure of each level in $C$. We define the column map $T_C$ to map a point in the level $i$ , $0\leq i\leq h-2$, to the point directly above in the level $i+1$. Thus, if we let $J$ to be the bottom level of the column $C$, then the $i$th level will be $T_C^i(J)$. $T_C$ brings the bottom $h-1$ levels to the top $h-1$ levels. A cutting and stacking construction for a measure-preserving transformation $T:X \to X$ consists of a sequence of columns $C_n=\left\lbrace J_n,...,T^{h_n-1}J_n\right\rbrace $ of height $h_n$ such that:

(i) $J_n$ is a disjoint union of elements from $\left\lbrace J_{n+1},...,T^{h_{n+1}-1}J_{n+1}\right\rbrace $.

(ii) $C_{n+1}$ is obtained from $C_n$ by cutting $C_n$ into $r_n\geq 2$ subcolumns $C_n^i$ of equal measure, putting a number of spacers (new levels of the same measure as any of the levels in the $r_n$ subcolumns) above each subcolumn, mapping the top level of each subcolumn to the spacer above it, and stacking left under right. In this way, $C_{n+1}$ consists of $r_n$ copies of $C_n$, possibly separated by spacers, which are also included in $C_n$.

(iii) $\bigcup_n  C_n$ is a generating subalgebra of the Borel sets $\mathcal{B}$. In our case, this usually requires that there be enough spacers for $\bigcup_n  C_n$ to have infinite measure.

(iv) The pointwise limit of $T_{C_n}$ as $n$ increases is $T$. \\

To discuss these cutting and stacking constructions, we introduce the idea of {\bf descendants}. Given a level $J$ $(J\in C_n)$ and any column $C_m$ where $m\ge n$, we define the $m$-descendants of $J$ to be the collection of levels in $C_m$ whose disjoint union is $J$. We denote this set by $D(J, m)$. Occasionally, we will also use $D(J, m)$ to refer to the heights of the $m$-descendants of $J$. (For more results using this notion of descendants, see \cite{SMALL}.) An important observation is the following:

\begin{lemma}
Suppose that $T$ is a rank-one transformation on an infinite measure space $X$, and that $J$ is a level of the $j$th column. Then $\mu(J\cap T^nJ)>0$ if and only if there exists $m$ such that $n \in D(J,m)-D(J,m)$. 
\end{lemma}

\begin{proof}

First, suppose that $\mu(J \cap T^n J)>0$ for some $n>0$. Choose $m$ such that $\max(D(J,m)) + n \leq h_m$. Then we can write $J= \cup_{D(J,m)}J^{l}$, where the union is after all the $m$-descendants of $J$. Further, write $J^{l}= T^{h(l)}I$, where $I$ is the bottom level of $C_j$. Then the condition $\mu(J\cap T^nJ)>0$ becomes $$\mu((\cup T^{h(l)}I)\cap (\cup T^{n+h(l)}I))>0,$$ which implies that $D(J,m) \cap (n+D(J,m)) \neq \emptyset$, that is, $n \in D(J,m)-D(J,m)$.
\end{proof}

Observe that actually we can compute the exact value of $\mu(J\cap T^nJ)$ by using this notion of descendants. Indeed, for large enough $N$ there will be at least $n$ spacers on top of the last column and therefore $T^n$ will be defined in $C_N$ on each of the descendants of $J$. Thus, 

\begin{corollary}
Suppose that $T$ is a rank-one transformation on an infinite measure space $X$, and that $J$ is a level of the $j$th column. Then if $\mu(J\cap T^nJ)>0$, then for all large $N$, $$\mu(J\cap T^nJ)=\frac{|  D(J,N)\cap (D(J,N)+n)|}{|D(J,N)|} .$$
\end{corollary}

\definition{A transformation $T$ is $v$-$\alpha$-type if for all sets $A$ of positive (finite) measure,$$\limsup \mu(A\cap T^{v_1n}A \cap \ldots \cap T^{v_d n}A)= \alpha \mu(A).$$}
This definition is inspired by \cite{HK64} and is useful for proving results allowing us to move from a sufficient semi-ring of intervals to all sets of positive measure. In fact, considering the larger class of transformations which have an $\alpha$ such that any set $A$ of positive (finite) measure has $$\limsup \mu(A\cap T^{v_1n}A \cap \ldots \cap T^{v_d n}A)\ge \alpha \mu(A)$$ gives us similar results, but we stick to the precedent. The concept of $\alpha$-type is also discussed more in a recent dissertation \cite{L09}.

Next, we have a basic analytic lemma. The exact statement of this lemma is from \cite{SMALL01}. 

\lemma{Double Approximation Lemma}\\
Let $\epsilon>0, \delta>0, 1>\tau>0$ and $I$ be an interval $\tau$-full of a measurable set $A$. If $\{r_n\}$ is an infinite sequence such that $r_n>1$ for large enough $n$, there exists an $N\in\mathbb{N}$ such that if $I_k=[\frac{k}{r_1\ldots r_N},\frac{k+1}{r_1\ldots r_N}]$ for $0\leq k\leq r_1\ldots r_N-1$, then there exists a subset $K\subset \{0,1,\ldots, (r_1\ldots r_N)-1\}$ with $|K|>(\tau-\delta)r_1\ldots r_N$ such that each $I_j, j\in K$, is $(1-\epsilon)$-full of $A$.\\

The next lemma is critical for verifying all of our later $v$-positive type examples.
\lemma{If $T$ is a rank-one transformation that is $v$-$\alpha$-type for all intervals that appear in a level of any of the defining columns, then $T$ is $v$-positive type.}
\begin{proof}
If $I, T^{v_1n}I,\ldots, T^{v_dn}I$ are in $C_N$ and are each $(1-\delta)$ full of $A$, then $I\cap T^{v_1n}I \cap \ldots \cap T^{v_d n}I$ is $(1-\delta (d+1))$-full of $A$. Then, given $I$ that is $(1-\delta)$-full of $A$, we can choose $\tau=(1-2\delta)$ so that beyond an N, $\tau$ of the subintervals of $I$ which appear as its descendants in later columns will be $(1-\delta)$-full. Then, it is a simple calculation to see that if $n$ is chosen so that $\mu(I\cap T^{v_1n}I\ldots \cap T^{v_dn})>(1-\epsilon)\alpha\mu(I)$, then the descendants will give that $$\mu(A\cap\ldots \cap T^{v_dn}A)\geq (1-\epsilon)(\alpha - 2\delta(d+1))(1-\delta(d+1))\mu(I)$$ 
Now, if we simply approximate all of $A$ with intervals $I$ that are $(1-\delta)$ of $A$ (choosing $\delta$ small enough so that $2\delta(d+1)<\alpha$), we have shown that $A$ has $v$-positive type. Of course we can do so and the result follows. 
\end{proof}

\lemma{If $T$ is a transformation that is $v$-$\alpha$-type along a fixed sequence (i.e. the $\limsup$ occurs along the same subsequence) for all sets in a sufficient semi-ring, then $T$ is $v$-$\alpha$-type for all sets.}\rm\\
The proof is essentially the same as the one given in \cite{S08}. This lemma will be the one used in all of the examples in this section, where the sufficient semi-ring is the set of all levels and the common subsequence will usually be some variant of $\{ h_n \}$.  

\lemma{If $T$ is $v$-zero type for all columns $C_m$, then $T$ is $v$-zero type.}\\ \rm \label{vzero}
If $T$ is infinite measure, the columns grow so that $X=\cup_m C_m$. For any set $A$ of positive finite measure and any $\epsilon>0$, we can pick $M$ such that $(A\cup\, ^cC_M)<\epsilon$. Then, $$\mu(A\cap T^{v_1n}A\ldots\cap T^{v_dn}A)< \mu(C_M\cap T^{v_1n}C_M \ldots \cap T^{v_dn}C_M) + (d+1)\epsilon$$
From this, we see that in fact $A$ has $v$-zero type. \\ \\

We can now exhibit examples of rank-one transformations that are $v$ positive type and others that are not $v$ positive type, which we do in sections four and five.

Note what these basic lemmas do not allow us to do. 
Moving from a sufficient semi-ring (usually composed of unions of intervals) to all finite positive-measure sets for $v$-positive type transformations is difficult without the additional control of $v$-$\alpha$-type.

\section{Rank-one Examples}\label{S:PTexamples}
Many of these examples are most easily seen visually, and we encourage the drawing of diagrams such as the ones shown below.

\example{For every vector $(v_1, \ldots, v_d)$, we can create a transformation $T$ that is $v$-positive type very simply. }\rm

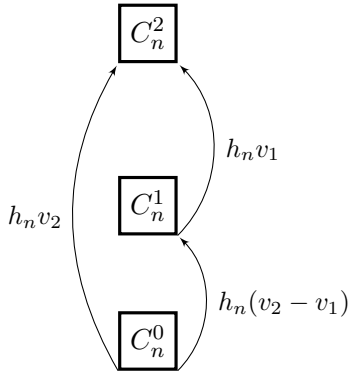
\begin{figure}[h]
\begin{tikzpicture}[->,>=latex',shorten >=1pt,auto,node distance = 1.8 cm]

\tikzstyle{clump}=[draw,very thick, fill=white, rectangle, font={\sffamily\small},minimum width=0.5 cm, minimum height=0.75 cm];
\tikzstyle{spacer}=[draw, very thick, fill, rectangle, minimum width=0.5 cm, minimum height=0.75 cm];
\tikzstyle{arrow}=[black, stateEdge, ->]
\tikzstyle{label} = [pos=0.5, text centered, font=\footnotesize];

     \draw  node[clump] (2)  {$C_n^2$};
     \draw  node[clump] (1) [below=1.5 cm of 2] {$C_n^1$} ;
     \draw node[clump] (0) [below of=1] {$C_n^0$};

     \path[->]
     	(1.south east) edge[bend right=45] node[label, xshift=2.5 em]{$h_nv_1$} (2.south east)
	(0.south west) edge[bend left=30] node[label, xshift=0 em]{$h_nv_2$} (2.south west)
	(0.south east) edge[bend right=45] node[label,xshift=5 em]{$h_n(v_2-v_1)$} (1.south east);     

\end{tikzpicture}
\caption{d=2 example\label{1}}
\end{figure}

\noindent An example for $d=2$ is given above. Let $T$ be a cutting and stacking transformation starting with $C_0$ consisting of one level: $[0,1]$. Then, for infinitely many $n$, cut the column $C_n$ into $d+1$ subcolumns and for $0\leq i\leq d-1$, on top of the $d-i$th subcolumn (called $C_n^{d-i+1}$) put $h_n(v_{i+1} - v_i-1)$ spacers ($v_0 = 0$).  The effect of the procedure is to put the first level of the $(d-i)$th subcolumn $h_n v_i$ levels away from the first level of the $(d+1)th$ subcolumn. 
As long as this is done infinitely often, at $\{n_k\}$, then $\mu(I\cap T^{h_{n_k}v_1}I \cap \ldots \cap T^{h_{n_k}v_d}I)> \frac{\mu(I)}{d+1}$, when $I$ is any level of any of the defining columns of $T$. By the lemma above, as this gives that $T$ is $v$-$\frac{1}{d+1}$-type for all of these intervals, we have that in fact $T$ is $v$-positive for all sets of positive measure. \\
\indent
This argument is very flexible, as we can add however many spacers we want after the last subcolumn and as we need only perform this procedure infinitely often, not at every step. For example, the construction given in \cite{SMALL99} describes how to make a rank-one power weakly mixing transformation with a procedure that needs similarly only occur infinitely often, and this procedure could be combined with the procedure in this example to create a rank-one transformation that is both power weakly mixing and $v$-positive type. Furthermore, as the set of finite integer-valued vectors is countable, we can use this same basic construction to create a rank-one transformation that is $v$-positive type for each $v\in \mathbb{Z}^d$ (for all values of $d>1$) and additionally power weakly mixing. \\

Next, we show that $v$-positive type and $v$-multiplicative-positive type are not equivalent. Before tackling a general example, here's a specific one.
\example{Let $v=(1,2)$. We construct a rank-one transformation $T$ that is $v$-multiplicatively positive type (hence positive type) but not $v$-positive type.}
\rm\\ \indent Start out with $C_0 = I = [0,1]$. At each step cut into 4 subcolumns. On the first subcolumn, put no spacers. On the next column, put $3h_n$ spacers. On the third subcolumn put $h_n$ spacers. On the last subcolumn put $8h_n$ spacers. Now, this is clearly $v$-multiplicative type for $I$ (just take overlap $T^{v_im}$ for $m=h_n$) and similarly for all interval levels. As the construction has the overlap occurring at $\frac{1}{4} \mu(I)$, $T$ is $v$-multiplicative-positive type for all positive measure sets. See Figure \ref{basic}.\\ \\

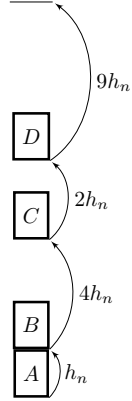
\begin{figure}[h]
\resizebox{!}{5.5cm}{
\begin{tikzpicture}[-,>=latex',shorten >=1pt,auto,node distance = 1.8 cm]

\tikzstyle{clump}=[draw,very thick, fill=white, rectangle, font=\footnotesize,minimum width=0.5 cm, minimum height=0.75 cm];
\tikzset{blank/.style={
        draw=none, rectangle, minimum width=0.7 cm, minimum height = 0 cm,
        append after command={
            [very thick,shorten >=0.2bp, shorten <=0.2bp]
            (\tikzlastnode.south west)edge[-](\tikzlastnode.south east)}}}

\tikzstyle{arrow}=[black, stateEdge, ->]
\tikzstyle{label} = [pos=0.5, text centered, font=\footnotesize];

	\draw node[blank] (4) {};
	\draw node[clump] (3) [below=1.8 cm of 4] {$D$};
     \draw  node[clump] (2)  [below=0.5 cm of 3] {$C$};
     \draw  node[clump] (1) [below=1 cm of 2] {$B$} ;
     \draw node[clump] (0) [below=0 cm of 1] {$A$};
     
     \path[->]
     	(3.south east) edge[bend right=55] node[label, xshift=2 em]{$9h_n$} (4.south east)
     (2.south east) edge[bend right = 50] node[label, xshift=2 em]{$2h_n$} (3.south east)
     	(1.south east) edge[bend right=45] node[label, xshift=2 em]{4$h_n$} (2.south east)
	(0.south east) edge[bend right=45] node[label,xshift=1.5 em]{$h_n $} (1.south east);     
\end{tikzpicture}
}
\caption{The basic transformation\label{basic}}
\end{figure}

 \indent To show $T$ is not $v$-positive type, we show this for $I$. For each $m$, consider $n$ where $h_{n-1} < m\leq h_n=16h_{n-1}$. Then, to show that $\mu(I\, \cap \,T^mI\, \cap\, T^{2m}I)=0$, we look at the four subcolumns $A, B, C, D$ of $C_{n-1}$ (not including spacers). It is clear that any nonzero intersection of $I, T^m I$, and $T^{2m} I$ will happen between two or more of the above subcolumns but in no case purely within one subcolumn. Therefore, it is just a matter of tracking down the different possible combinations of subcolumns and showing that there can be no overlap. We show these different scenarios with corresponding pictures. To start, we note that $\mu(D \,\cap\, T^m C \,\cap\, T^{2m} I)=0$ as levels in $D$ and $C$ that are both in $I$ are at most $2h_{n-1}+8h_{n-2}$ apart and levels in $D$ and $B$ that are both in $I$ are at least $5h_{n-1}+8h_{n-2} > 2(2h_{n-1} + 8h_{n-2})$ apart. Thus no possible $m$ can give this type of overlap. See Figure \ref{2}. \\ 
 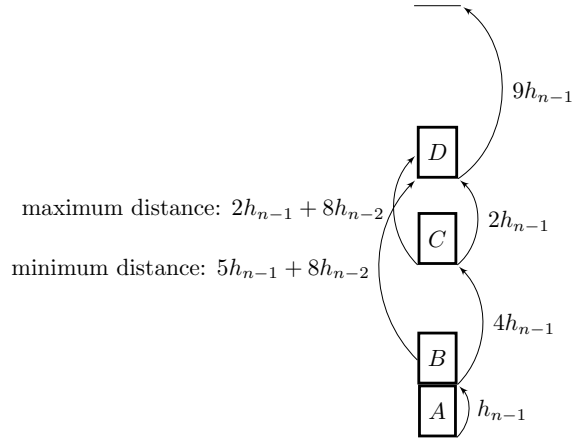
\begin{figure}[h]
 \resizebox{!}{6cm}{
\begin{tikzpicture}[-,>=latex',shorten >=1pt,auto,node distance = 1.8 cm]

\tikzstyle{clump}=[draw,very thick, fill=white, rectangle, font=\footnotesize,minimum width=0.5 cm, minimum height=0.75 cm];
\tikzset{blank/.style={
        draw=none, rectangle, minimum width=0.7 cm, minimum height = 0 cm,
        append after command={
            [very thick,shorten >=0.2bp, shorten <=0.2bp]
            (\tikzlastnode.south west)edge[-](\tikzlastnode.south east)}}}

\tikzstyle{arrow}=[black, stateEdge, ->]
\tikzstyle{label} = [pos=0.5, text centered, font=\footnotesize];

	\draw node[blank] (4) {};
	\draw node[clump] (3) [below=1.8 cm of 4] {$D$};
     \draw  node[clump] (2)  [below=0.5 cm of 3] {$C$};
     \draw  node[clump] (1) [below=1 cm of 2] {$B$} ;
     \draw node[clump] (0) [below=0 cm of 1] {$A$};
     
     \path[->]
     	(3.south east) edge[bend right=55] node[label, xshift=3 em]{$9h_{n-1}$} (4.south east)
     (2.south east) edge[bend right = 50] node[label, xshift=3 em]{$2h_{n-1}$} (3.south east)
     	(1.south east) edge[bend right=45] node[label, xshift=3 em]{4$h_{n-1}$} (2.south east)
	(0.south east) edge[bend right=45] node[label,xshift=2.5 em]{$h_{n-1} $} (1.south east)
	(1.mid west) edge[bend left=45] node[label, xshift=0 em]{minimum distance: $5h_{n-1}+8h_{n-2}$} (3.south west)
	(2.south west) edge[bend left=45] node[label, xshift=0 em]{maximum distance: $2h_{n-1}+8h_{n-2}$} (3.mid west);

\end{tikzpicture}
}
\caption{Disproving nontrivial $D \,\cap\, T^m C \,\cap\, T^{2m} I$ overlap\label{2}}
\end{figure}

 \indent
For the second scenario, we consider $B \cap T^mA \cap T^{2m}I$ overlap.
Although $m\le h_n$, we might have overlap occurring in a larger picture, as illustrated below with Figure \ref{3x}. Here the maximum distance available ($h_{n-1} + 8h_{n-2}$) is even smaller than that considered in the first case and the minimum distance ($h_{n-1} + 8h_{n-2}$) is even larger, so the same argument as before shows that there must be zero overlap incorporating a smaller subcolumn.

  \begin{figure}[h]
    \resizebox{!}{5.5cm}{
\begin{tikzpicture}[-,>=latex',shorten >=1pt,auto,node distance = 1.8 cm]

\tikzstyle{clump}=[draw,very thick, fill=white, rectangle, font=\footnotesize,minimum width=0.5 cm, minimum height=0.75 cm];
\tikzset{blank/.style={
        draw=none, rectangle, minimum width=0.7 cm, minimum height = 0 cm,
        append after command={
            [very thick,shorten >=0.2bp, shorten <=0.2bp]
            (\tikzlastnode.south west)edge[-](\tikzlastnode.south east)}}}

\tikzstyle{arrow}=[black, stateEdge, ->]
\tikzstyle{label} = [pos=0.5, text centered, font=\footnotesize];

     \draw  node[clump] (2)  {$C$};
     \draw  node[clump] (1) [below=1 cm of 2] {$B$} ;
     \draw node[clump] (0) [below=0 cm of 1] {$A$};
     \draw node[clump] (-1) [below=2 cm of 0] {$D'$};
     
     \path[->]
     	(1.south east) edge[bend right=45] node[label, xshift=3 em]{4$h_{n-1}$} (2.south east)
	(0.south east) edge[bend right=45] node[label,xshift=2.5 em]{$h_{n-1} $} (1.south east)
	(-1.mid west) edge[bend left=45] node[label, xshift=0 em]{minimum distance: $8h_{n-1}+8h_{n-2}$} (0.south west)
	(0.south west) edge[bend left=45] node[label, xshift=0 em]{maximum distance: $h_{n-1}+8h_{n-2}$} (1.mid west)
	(-1.south east) edge[bend right=55] node[label, xshift=3 em]{$9h_{n-1}$} (0.south east);

\end{tikzpicture}}
\caption{Disproving nontrivial $B \cap T^mA \cap T^{2m}I$ overlap\label{3x}}
\end{figure}

\indent
 The impossibility of the third scenario, requiring $\mu((C \cup D) \cap T^m(A\cup B) \cap T^{2m}I)=0$, follows very similarly and is shown in Figure \ref{5}. 
 
   \begin{figure}[h]
     \resizebox{!}{6.5cm}{
\begin{tikzpicture}[-,>=latex',shorten >=1pt,auto,node distance = 1.8 cm]

\tikzstyle{clump}=[draw,very thick, fill=white, rectangle, font=\footnotesize,minimum width=0.5 cm, minimum height=0.75 cm];
\tikzset{blank/.style={
        draw=none, rectangle, minimum width=0.7 cm, minimum height = 0 cm,
        append after command={
            [very thick,shorten >=0.2bp, shorten <=0.2bp]
            (\tikzlastnode.south west)edge[-](\tikzlastnode.south east)}}}

\tikzstyle{arrow}=[black, stateEdge, ->]
\tikzstyle{label} = [pos=0.5, text centered, font=\footnotesize];

	\draw node[clump] (3){$D$};
     \draw  node[clump] (2)  [below=0.3 cm of 3] {$C$};
     \draw  node[clump] (1) [below=0.6 cm of 2] {$B$} ;
     \draw node[clump] (0) [below=0 cm of 1] {$A$};
     \draw node[clump] (-1) [below=2 cm of 0] {$D'$};
     
     \path[->]
     (2.south east) edge[bend right = 50] node[label, xshift=3 em]{$2h_{n-1}$} (3.south east)
     	(1.south east) edge[bend right=45] node[label, xshift=3 em]{4$h_{n-1}$} (2.south east)
	(0.south east) edge[bend right=45] node[label,xshift=2.5 em]{$h_{n-1} $} (1.south east)
	(-1.mid west) edge[bend left=45] node[label, xshift=0 em]{minimum distance: $8h_{n-1}+8h_{n-2}$} (0.south west)
	(0.south west) edge[bend left=45] node[label, xshift=0 em]{maximum distance: $7h_{n-1}+8h_{n-2}$} (4.mid west)
	(-1.south east) edge[bend right=55] node[label, xshift=3 em]{$9h_{n-1}$} (0.south east);

\end{tikzpicture}}
\caption{Disproving nontrivial $(C \cup D) \cap T^m(A\cup B) \cap T^{2m}I$ overlap\label{5}}
\end{figure}
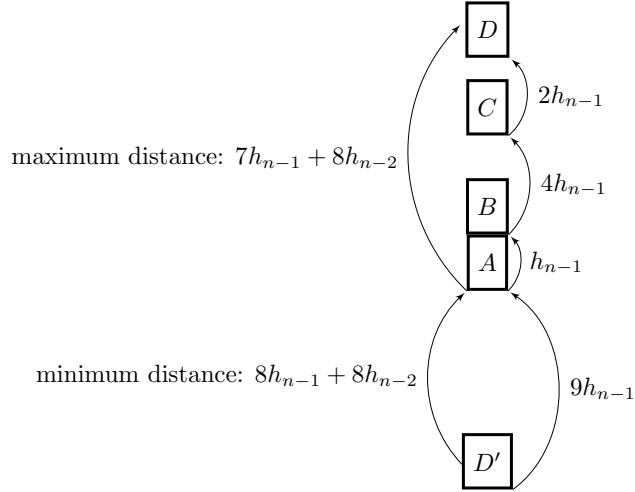
 
Tracking down each of these scenarios gives a contradiction requiring the intersection to have measure 0, and so we conclude that for $h_{n-1}\le m \le n$,  $\mu(I\,\cap\, T^m I \,\cap\, T^{2m}I)=0$. Unlike the previous example, there is little flexibility, as if one were to perform cutting and stacking steps of a different type, one would need to check that these additional steps do not incidentally bestow $v$-positive type. However, 
note that the proof above actually shows that $T$ is $v$-positive type for no $A$ of finite measure as we can choose a finite union of levels $J\subset C_n$ such that $A\subset J$ (within $\epsilon$) and then for $m>h_n$, we have $\mu(J\cap T^mJ\cap T^{2m}J)=0$ by the argument above.\\ \\ 

Now, the above example generalizes to the following example: \example{For every vector $v$ with $\frac{v_d}{v_1}\geq d$, we can create a rank-one transformation $T$ that is $v$-multiplicative-positive type but not $v$-positive type. }\rm \\ \indent This transformation is in a spirit similar to the previous one. To create $T$ with cutting and stacking, let $C_0 = [0,1]$. At \emph{each} \rm step, cut into $2d$ subcolumns. Then, on top of the $2i-1$th subcolumn $1\leq i\leq d$, place $(v_i -1)h_n$ spacers. On top of the $2i$th subcolumn ($i<d$), place $M$ spacers, where $M\geq v_d/v_1(v_d+1)v_d$ and $M\geq d(v_d+1)v_1/(v_d - v_1(d-1))$. On top of the last column, put $kh_n$ spacers so that the spacers added on top of that last subcolumn have measure equal to the sum of the measures of each subcolumn and of the spacers put on top of each of the previous subcolumns. It is fairly clear that $T$ is $v$-multiplicative type. Whenever $n=h_m$, the desired inequality will be achieved for any level interval. And by the lemma above that means $T$ is $v$-multiplicative type for any set of positive measure.\\ \indent To see that $T$ is not $v$-positive type, we primarily concentrate on the beginning level $I=C_0$. (Note that a failure for this interval will imply a failure for all other levels but perhaps not for other sets.) For any $m$, pick $n$ so $h_{n-1}<m\leq h_n$. Then, considering subcolumns $X_{2i}$ and $X_{2i-1}$, we note that if $\mu(X_{2i}\cap T^{mv_1}X_{2i-1})>0$, then the maximum value for $v_1m$ is $v_dh_{n-1}+kh_{n-2}\leq (v_d+1)h_{n-1}$ and then $T^{-v_dm}X_{2i}$ will be entirely contained within the spacers above $X_{2i-2}$. Then, if $\mu(X_i \cap T^{v_1m}X_j)>0$ for $X_i, X_j$ with at least $Mh_{n-1}$ spacers between them, we note that the maximum distance between two levels contained in $I$ is $h_{n-1}(M(d-1) + d(v_d+1))$ and so $T^{-v_dm}X_i$ is contained in more spacers, by the choice of $M$ detailed above. \\ \\

In fact, the above two examples are examples of rank-one transformations that are clearly positive type (as they are $v$-multiplicatively-positive type), but in fact we can show that they are $v$-zero type. The argument above showed that for $I=C_0$, $$\mu(I \cap T^{v_1 n} I \cap \ldots \cap T^{v_d n} I) \to 0$$

This argument generalizes easily to $I=C_k$ for any $k$. In fact, the argument given holds if we confine ourselves $m\ge h_N$ where $N\ge k$, and so the limit for any $I=C_k$ will still be zero, and by Lemma \ref{vzero}, we can conclude that $T$ is $v$-zero type.

\section{Multiple Recurrence and Positive Type}\label{S:Markov}
Just as the rank-one examples helped us distinguish $v$-positive type from $v$-multiplicative-positive type, Markov shift examples allow us to distinguish between $v$-recurrence and $v$-positive type. Multiple recurrence is a property famously studied in the finite measure case, culminating in Furstenberg's Multiple Recurrence Theorem. More recently, it has been studied in the infinite measure case, the reader may refer to \cite{AN00} for Markov shift examples and \cite{EHH98} for odometer (rank-one) examples. 

\begin{definition}
A transformation $T$ is $k$-recurrent if for every measurable set  $A$ of positive measure, there exists an integer $n>0$ such that \[\mu(A \cap T^{n} A \cap \ldots \cap T^{kn}A )>0\] and is multiply recurrent if $T$ is $k$-recurrent for every positive integer $k$. 
\end{definition}

Analogously, we define $v$-recurrence as the notion that for a given positive measure set $A$ and vector $v$, there exists an $n>0$ such that $\mu(A\cap T^{v_1n}A \cap\ldots \cap T^{v_d n}A)>0$. We must conclude that $v$-recurrence and $v$-positive type, although the latter implies the former, are different notions. An example showing this is given by a conservative ergodic Markov shift  such that $T\times T$ is conservative ergodic, while $T\times T\times T$ is not conservative \cite{KP63}. Then $T$ cannot have positive type, otherwise, by  \cite[Proposition 2.2]{AN00}, $T\times T\times T$ will have also positive type, so, in particular, it will be conservative (or more generally, such a  $T$ is remotely infinite  \cite{BF64}, so of zero type \cite{KS69}). $T$ is therefore not positive type and thus clearly not $(1,2)$ positive type. However, by  \cite[Theorem 1.1]{AN00}, $T\times T$ conservative is equivalent, for a general Markov shift $T$, to $(1,2)$ recurrence. This means that this above example is $(1,2)$ recurrent but not $(1,2)$-positive type.\\

\section{Results for More General Sequences}
This section contains some basic facts that encourage us to believe that the above definitions for $v$-positive type and $v$-multiplicative-positive type are some of the most natural properties to consider. Specifically, this section explores the quantities $\mu(A\cap T^{n_i} A)$ for a variety of different, more general sequence $\{n_i\}$.\\ 

First, we present a simple result indicating that considering sequences other than $\{kn\}$ may prove difficult: 
 \fact{For any $h$ there is an ergodic $T$ of positive type, measure preserving and invertible, such that for some $A$ of positive measure and for any $c$ not sharing all prime factors of $h$, $\lim_{n\to\infty} \mu(T^{hn+c}A\cap A ) = 0$.\\
 \indent To see this, we will construct a generalization of the 2-point extension. First, pick a prime factor $k$ of $h$ such that $k$ does not divide $c$. Then let $S:X\to X$ be any measure-preserving invertible positive type transformation. (Example: the canonical Hajian-Kakutani skyscraper.) Then, create $k$ copies of this measure space, labeling them $X_i$ and using the same coordinate system on each. Define $T$ as follows: for any $x_1\in X_1$, let $T(x_1)=x_1\in X_2$, using the same coordinates as mentioned. Similarly, continue up to $T^{k-1}(x_1)=x_1\in X_k$ and then let $T^k(x_1)=S(x_1)\in X_1$. Because $k$ is prime, $T$ is conservative ergodic by S's ergodicity. Then, if $A$ is a positive measure set so that $A\subset X_1$, it is very clear that for any $n$,$\mu(T^{kn+c}A\cap A )$, and so the same assertion with $h$ replacing $k$ must also hold. \\

 \fact{For any $k$ there exists $T$, $A$, $c$ such that$$\limsup \mu(A\cap T^nA\cap T^{kn}A) >0 \text{ but } \lim \mu(T^cA\cap T^nA\cap T^{kn}A)=0$$} To obtain an example, just take a rank-one construction where $r_n=k$ and add no stackers infinitely often to get the first inequality. However, pick $c$ to be some integer that is not expressible as $\sum_{n\leq N} a_n h_n$ where each $a_n\in \{-(k-1), \ldots (k+1)\}$. If we choose $h_n$ to be increasing quickly enough, there are some $c$ fitting this description. \\ 
  
Given these two facts, it is natural to restrict to considering intersections merely in the form $T^{kn}A\cap A$ without additional constants. Intersections of this form are also already studied in the context of recurrence, so this further seemed a natural choice. However, when still considering sequences in general, we did obtain the following two results:

\proposition{If  $\lim_{n\to \infty} \mu(A\cap T^{a_n}A\cap T^{a_n+b_n}A) = 0$ for every pair of infinite sequences $\{a_n\}, \{ b_n \}$ and every set $A$ of finite positive measure, then $T$ is zero type. } 
\begin{proof} The proof proceeds by using the contrapositive. Assume that $T$ is of positive type. Specifically, for any $A$ of positive, finite measure we can pick a sequence $\{ a_k \}$ so that  $$\lim_{n\to \infty} \mu(A\cap T^{a_n}A) = c> 0.$$ Then, we need to simply pick $\{ b_n \}$ and $\{a'_n\}$ such that $$\lim_{n\to \infty} \mu(A\cap T^{a'_n}A\cap T^{a'_n+b_n}A) > 0.$$ To see that this is possible, note that we may assume that $\mu(T^{a_n}A \cap A) \geq c'=\mu(A)/M$ for some large $M$. Then, for any beginning $N$, we note that $\mu(\bigcup_{n=N}^{N+2M} A\cap T^{a_n}A) \leq \mu(A)$ and additionally 
$$\mu(\bigcup_{n=N}^{N+2M} A\cap T^{a_n}A)\geq $$  $$ \sum_{n=N}^{N+2M}\mu(A\cap T^{a_n}A) \, -\sum_{N\leq n,m\leq N+2M} \mu(A\cap T^{a_n}A\cap T^{a_m}A)\geq $$
$$ 2\mu(A) \, -\sum_{N\leq n,m\leq N+2M} \mu( A\cap T^{a_n}A\cap T^{a_m}A)$$
 and
 so $$\sum_{N\leq n,m\leq N+2M} \mu(A\cap T^{a_n}A\cap T^{a_m}) \geq \mu(A)$$ and then there exist some $n,m$, $n\geq m$ in the given range such that $\mu( A\cap T^{a_n}A\cap T^{a_m}A) \geq \mu(A) / \binom{2M+1}{2}$ and then we can pick $a'_i=a_m$ and $b_i = a_n - a_m$ and then we can choose successive terms $a_{i+1}$ and $b_{i+1}$ by increasing $N$. Then, consequently, $$\limsup_{n\to \infty} \mu(A\cap T^{a'_n}A\cap T^{a'_n+b_n}A ) \geq \mu(A) / \binom{2M+1}{2}> 0.$$

\end{proof}

\proposition{ For any pair of infinite sequences $\{a_n\}, \{ b_n \}$, there is a $T$ of positive type such that  $\lim_{n\to \infty} \mu(A\cap T^{a_n}A\cap T^{a_n+b_n}A) = 0$}. In fact, this $T$ can be chosen to be a half-rigid rank-one transformation.
\begin{proof}
The transformation constructed will be a variant of the Hajian-Kakutani skyscraper. At each step of construction, we divide the column into two equal pieces each of height $h_n$ and stack an additional $c_n$ pieces on top of the rightmost column, where $c_n$ is to be determined. Then, we stack the columns, left under right, and continue. Thus, $h_{n+1} = 2h_n + c_n$ and so our only task is to pick a sequence $\{ h_n \}$ so that the transformation satisfies the above condition. We note that if $I=[0,1]$ is the initial column in the 0th step, then $\mu(I\cap T^{h_n} I) \ge \frac{1}{2}\mu(I)$ and so $T$ is clearly of positive type. In fact, in the nth column, the levels that are contained in $I$ occur at the following heights: $0, h_1, h_1, h_2, h_1+h_2, h_3, \ldots h_n+h_{n-1}+ \ldots h_1$ and therefore moving up by $h_n$ moves the bottom half of these levels up to the top half of these intervals. Let $I_n$ represent the levels in $I$ that occur in the $n$th column $C_n$.\\

\noindent Now, we pick $I=[0,1]$. Then $\lim_{n\to \infty} \mu(I\cap T^{a_n}I\cap T^{a_n+b_n}I) = 0$ will hold if for each $n$, the difference set of the descendants of $I$ in the $N$ column -- i.e. the set $\Delta(I, N) = D(I,N)- D(I,N)$ of numbers, each expressible as $\sum_{i=0}^N c_i h_i$ where each $c_i \in \{ -1, 0, 1\}$ -- contains for any $k$ at most two of $\{a_k, b_k, a_k+b_k\}$. For example, if for a particular $k$ the set $\Delta(I,N)$ contains $a_k$ and $a_k+b_k$ but not $b_k$, then $\mu(I \cap T^{a_k}I) \neq 0$ and $\mu( I\cap T^{a_k+b_k}I) \neq 0$ but $\mu(T^{a_k}I\cap T^{a_k+b_k}I)=\mu(I \cap T^{b_k}I)=0$ by the above discussion, and so $\mu( I \cap T^{a_k}I\cap T^{a_k+b_k}I)=0$. We made this argument for $I$, but to expand this conclusion to a general set $A$ of finite measure, we need to make this argument for any column $C_m$. Numbers in $D(C_m, N)-D(C_m, N)$ are not of the form $\sum_{i=0}^N c_i h_i$; instead they are of the form $\sum_{i=m}^N a_i h_i + b$ where $b$'s only restriction is that $|b|\le h_m$. So it is impossible to guarantee the given statement for $k$ with $a_k\le h_m$. However, because we care about the $\lim_{n\to \infty}$, we concern ourselves with $k$ large enough that $a_k, b_k \ge h_m$. So long as there are at most two of $\{a_k, b_k, a_k + b_k\}$ for these large $k$, the above argument gives us the statement for $C_m$. In other words, we will be operating at such a large scale that the column $C_m$ acts like a level. Now, it follows that the statement holds for any finite union of levels and thus, by an approximation argument, for any finite measure set $A$. 

To choose $h_n$ to avoid all of these individual terms, we first require that $h_n \geq 4h_{n-1}$ (i.e. $c_n \ge 2h_n$). This condition forces all numbers in $\Delta = \cup_{n} \Delta(I,N)$ to have a unique representation in the form $\sum_{i=0}^N a_i h_i$ where each $a_i \in \{ -1, 0, 1\}$. Additionally, it means that if $a_k\not\in \Delta(I, N-1)$ and $a_k \le 3h_N$, then $a_k \not\in \Delta(I, N)$. Relatedly, this tells us that if $a_k, b_k\in \Delta(I, N-1)$, either $a_k+ b_k\in\Delta(I, N-1)$ or $a_k +b_k \not\in\Delta(I, N)$. Then, if $a_k\in \Delta(I, N-1)$ but $b_k\not\in \Delta(I, N-1)$, we vary $c_{N-1}$, starting at $3h_{N-1}$ and increasing if necessary, so that $b_k\not\in\Delta(I,N)$. As we increase, we will certainly reach a value of $c_{N-1}$ such that $b_k\le 3h_N=3(2h_{N-1} + c_{N-1})$. But there are a variety of $k$ that we have to worry about. Specifically, let $$K_a = \{ k | a_k\in \Delta(I, N-1)\text{ but }b_k\not\in \Delta(I, N-1) \}$$ and let $$K_b = \{ k | b_k\in \Delta(I, N-1)\text{ but }a_k\not\in \Delta(I, N-1) \}.$$ Then both $K_a$ and $K_b$ are both finite and depend only on $h_i$, $i\le N-1$. As we vary $c_{N-1}$, these sets will not change. Additionally, they are both finite sets. That means that we can increase $c_N$ without worrying about adding to these sets. In particular, we can make $c_{N-1}$ big enough to miss all $a_k$ for $k\in K_b$ and $b_k$ for $k\in K_a$. Now, we might introduce new $a_j$, but either we introduce only $a_j$ or we introduce both $a_j, b_j\in \Delta(I, N)$ be increasing $c_{n-1}$. In the first case, we have not violated any conditions and we will deal with that value of $j$ when we move on to $c_n$ or larger. In the second case, as both $a_j, b_j\not\in\Delta(I, N)$, we may conclude that both $a_j, b_j \ge c_{N-1} \ge 2h_{n-1}$ so that  both $a_j$ and $b_j$ have $h_N$ in their descendent expansions. Thus, $a_j + b_j \ge 2h_{N} -2h_{N-1}\ge h_n + 2h_{N-1}$, telling us that $a_j + b_j \not\in \Delta(I, N-1)$. Repeating this process, it is clear that for no $k$ will all three of $\{a_k, b_k, a_k + a_k\}$ be in $\Delta$. Consequently, $\lim_{n\to \infty} \mu(A\cap T^{a_n}A\cap T^{a_n+b_n}A) = 0$.

\end{proof}}
In fact, the above constructed $T$ also has $$\lim_{n\to \infty} \mu(A\cap T^{a_n})\mu(A\cap T^{b_n}A)\mu(A\cap T^{a_n+b_n}A) = 0$$

The above examples show that $v$-positive type, due partly to its similarity to $k$-recurrence and partly due to the examples above, is a natural generalization to consider.

\section{Further explorations: positive type for different $v$}

We now want to discuss vectors $v,w$ in pairs. If $T$ is $v$-positive type, is it $w$-positive type? We start with a new definition for comparing vectors. Unfortunately, we will require different definitions for $v$-positive type and $v$-multiplicatively positive type. 

\begin{definition}
Given two vectors $v \in \mathbb{Z}^d$ and $w\in \mathbb{Z}^{d'}$, if there exist positive integers $n, m, c$ such that the set of components of $nv$ is included in the set of components of $mw-c$, where $c=0$ or $c=mw_i$ for some $i$, we say that $v$ is \emph{positive type less than} $w$, from now on denoted $v\leq_p w$. (Note that this property requires automatically that $d\le d'$).

Furthermore, if there exist positive integers $n,m,c$ such that each $nv_i =\sum_{j\in I(i)} mw_j - \sum_{j'\in J(i)}mw_{j'} $, then we say that $v$ is \emph{multiplicative-positive type less than} $w$, from now on denoted $v \leq_m w$.  Note that $v\leq_p w \implies v \leq_m w$ but the reverse direction does not hold. 
\end{definition}

One note about the above definitions: in the following theorems  using $\le_m$ and $\le_p$, we note that these constructions often hold only for vectors $v$ in the accepted form, that is $v_i$ distinct. Any transformation $T$ of positive type has $v$-positive type for any $v=(\ell, \ell, \ldots, \ell)\in \Zz^d$ or other such ``redundant vectors'' need to be considered specially from the vectors in standard format.

Observe also that if $v\leq_m w$, if $T$ is $w$ multiplicative-positive type, then $T$ is also $v$ multiplicative-positive type. Indeed, suppose that $nv$ has all its components contained in $mw$. If $T$ has $w$ multiplicative-positive type, then $T^m$ has $w$ multiplicative-positive type, so $T$ has $mw$ multiplicative-positive type, hence $T$ has $nv$ multiplicative-positive type. If we suppose that $T$ has $v$-multiplicative-zero type, then $T$ will have trivially $nv$-multiplicative-zero type. As this is a contradiction, $T$ has also $v$ multiplicative-positive type.

\begin{theorem}
Suppose $v\in \mathbb{Z}^d$. Then there exists a rank-one transformation $T$ which has $v$ multiplicative-positive type, but which does not have $w$ multiplicative-positive type for any $w\not\leq_m v$.
\end{theorem}

\begin{proof}
We will construct a rank-one transformation which satisfies the above conditions as follows. Let $I=C_0=[0,1]$. For each $m$, we construct $C_{m+1}$ by cutting $C_m$ in $d+1$ parts and adding  spacers on the first $d$ subcolumns such that the number of levels between the first $(d+1)$th of the bottom level of $C_m$ and the bottom $(d+1)$ths of the other levels are $$D':=\left\lbrace 0, v_1h_m, v_2h_m,...,v_dh_m\right\rbrace.$$ On the last subcolumn, we add a lot of spacers. We add sufficiently many spacers $s_n$ so that there is a unique representation of elements in $D(I,N)-D(I,N)$ in base $(h_i)$  and in fact so that where $H_m = \sum_{k\leq m} \sum_{1\leq i\leq d} v_i h_k$ we have the ratio $h_{m+1}/H_m$ grow without bound. Then we stack from left to right. See Figure \ref{whee} an illustration of this construction where $d=3$.\\ 
\begin{figure}  \resizebox{!}{6.5cm}{\begin{tikzpicture}[-,>=latex',shorten >=1pt,auto]
\tikzstyle{clump}=[draw,very thick, fill=white, rectangle, font=\footnotesize,minimum width=0.5 cm, minimum height=0.75 cm];
\tikzset{blank/.style={
        draw=none, rectangle, minimum width=0.7 cm, minimum height = 0 cm,
        append after command={
            [very thick,shorten >=0.2bp, shorten <=0.2bp]
            (\tikzlastnode.south west)edge[-](\tikzlastnode.south east)}}}

\tikzstyle{arrow}=[black, stateEdge, ->]
\tikzstyle{label} = [pos=0.5, text centered, font=\footnotesize];

	\draw node[blank] (4) {};
	\draw node[clump] (3) [below=1.8 cm of 4] {$C_3$};
     \draw  node[clump] (2)  [below=0.8 cm of 3] {$C_2$};
     \draw  node[clump] (1) [below=0.7 cm of 2] {$C_1$} ;
     \draw node[clump] (0) [below=0.5 cm of 1] {$C_0$};
     
     \path[->]
     	(3.south east) edge[bend right=55] node[label, xshift=1.5 em]{$s_n$} (4.south east)
     (2.south east) edge[bend right = 50] node[label, xshift=2.5 em]{$v_3h_n$} (3.south east)
     	(1.south east) edge[bend right=45] node[label, xshift=2.5 em]{$v_2h_n$} (2.south east)
	(0.south east) edge[bend right=45] node[label,xshift=2.5 em]{$v_1h_n $} (1.south east);     
\end{tikzpicture}}\caption{\label{whee}}\end{figure}
\indent
Now, to see that $T$ is $v$-multiplicative-positive type we note that for any $n=h_m$, we have (using the notion of descendants):

 \begin{align*} 
\mu(I\cap &T^{v_1h_m}I)...\mu(I\cap T^{v_dh_m}I)\\&\geq  \left(\frac{\mu(I)}{d+1} |D'\cap D'+v_1h_m|\right)\ldots \left(\frac{\mu(I)}{d+1} |D'\cap D'+v_dh_m|\right)\\ 
&\geq \left(\frac{1}{d+1}\right)^d |D' \cap D' + v_1h_m|\ldots |D' \cap D'+v_dh_m|\\ 
&= \left(\frac{1}{d+1}\right)^d.
 \end{align*} 
 
In particular, this proves that $$\limsup \mu(I\cap T^{v_1n}I)...\mu(I\cap T^{v_dn}I)>0.$$ By the lemmas in earlier sections, as we have this constant $\left(\frac{1}{d+1}\right)^d$ delineating the positive overlap, this argument gives $T$ is $v$-multiplicative-positive type for all $A$ of positive measure.\\ \\
Now, given any $w\in \mathbb{Z}^{d'}$, if for every $M$ there exists an $n>M$ such that $\mu(I\cap T^{v_1n}I)...\mu(I\cap T^{v_dn}I)> 0$, we can show that $w\leq_m v$ as described above. First, pick an $M$ large enough so that $$h_{M}>\frac{ w_{d'}}{w_1}H_{M-1}$$
Then, if we have $n>M$ and the minimum $N(n)$ such that  $$w_1n, w_2n,..., w_cn \in D(I,N)-D(I,N),$$ by the construction of the columns $C_n$, there is a unique representation of elements in $D(I,N)-D(I,N)$ in base $(h_j)$. For each $1\leq i \leq d'$ let $A_ih_{\ell_i}$ be the biggest term in the expansion of $w_in$. Then, for any $i,j$ derived this way, we note that because $h_{m+1} \gg h_m$ and by our choice of $M$, we must have $h_{\ell_i}=h_{\ell_j} = h_N$. (Thus $\ell_i=\ell_j=N$). In fact, we note that when we express $\frac{w_i}{w_j}$ using the expansions, the $h_N$ dominate all other terms. For large enough $n$, as $w_1n, \ldots, w_{d'}n$ must all occur in the same scale $(h_m \ge w_i n \ge h_{m+1}$), and as the construction gives the same method for cutting and stacking at every $n$, there are a finite number of possibilities for $A_i, A_j$ as we increase $n$. As $h_{m+1}/H_M$ grows without bound, for some choice of $A_i, A_j$, as we increase $n$ we see that  $\frac{A_i}{A_j} - \lim \frac{H_{n-1}}{h_n}\le\frac{w_i}{w_j}\le \frac{A_i}{A_j} + \lim \frac{H_{n-1}}{h_n}$ has the limit go to zero, we observe that the following ratio will be satisfied:
$$\frac{w_i}{w_j} = \frac{A_i}{A_j}\text{ or } w_i = A_i \big(\frac{w_1}{A_1}\big)$$
 However, each 
 \[A_i \in \left\{ 0, v_1, v_2, \ldots, v_d, v_2 - v_1, \ldots , v_{d-1} - v_d, \ldots, \sum_I v_j - \sum_J v_i \right\}\] 
 for certain partitions $I$ and $J$ by the construction of $T$. Substituting these terms in explicitly, we see that $w_i = (\sum v_k-\sum v'_k)\frac{w_1}{\sum v_j-\sum v'_j}\implies w_i \sum v_j - w_i\sum v'_j = \sum v_k w_1-\sum v_k'w_1$. Note that the $v_j,v'_j$ (derived from $A_1$) are constant and the $v_k,v_k'$ (derived from $A_i$) are not. Thus by definition, $w\leq_m v$. 
 
One note: not all $w\leq_m v$ will have $w$-multiplicative-positive type for this constructed $T$. However, variations $S$ on $T$ can be constructed so that $S$ will be $w$-multiplicative-positive type. These variations are obtained by putting the spacers on the subcolumns so that the order is varied -- $v_2 h_n$ before $v_1 h_n$, for example.  
 
\end{proof}
~\

\begin{theorem}
The same statement as Theorem 7.2 but for $v$-positive type, replacing $\le_m$ with $\le_p$. 
\end{theorem}

\begin{proof}
We construct a similar transformation, with the spacer pattern seen in Figure \ref{last}
\begin{figure}  \resizebox{!}{10.5cm}{\begin{tikzpicture}[-,>=latex',shorten >=1pt,auto]
\tikzstyle{clump}=[draw,very thick, fill=white, rectangle, font=\footnotesize,minimum width=1 cm, minimum height=0.75 cm];
\tikzstyle{dot} = [draw, circle, fill=black, scale=0.1];
\tikzset{blank/.style={
        draw=none, rectangle, minimum width=0.7 cm, minimum height = 0 cm,
        append after command={
            [very thick,shorten >=0.2bp, shorten <=0.2bp]
            (\tikzlastnode.south west)edge[-](\tikzlastnode.south east)}}}

\tikzstyle{arrow}=[black, stateEdge, ->]
\tikzstyle{label} = [pos=0.5, text centered, font=\footnotesize];

	\draw node[blank] (6) {};
	\draw node[clump] (5) [below = 2 cm of 6] {$C_{d}$};
	\draw node[clump] (3) [below=1.8 cm of 5] {$C_{d-1}$};
	\draw node[dot] (a) [below=0.3 cm of 3] {};
	\draw node[dot] (b) [below=0.2 of a] {};
	\draw node[dot] (c) [below=0.2 of b] {};
     \draw  node[clump] (2)  [below=1.1 cm of 3] {$C_2$};
     \draw  node[clump] (1) [below=0.7 cm of 2] {$C_1$} ;
     \draw node[clump] (0) [below=0.5 cm of 1] {$C_0$};
     
     \path[->]
     	(5.south east) edge[bend right=55] node[label, xshift=1.8 em]{$s_n$} (6.south east)
	(3.south east) edge[bend right=55] node[label, xshift=3 em]{$v_1h_n$} (5.south east)
     	(1.south east) edge[bend right=45] node[label, xshift=7 em]{$(v_{d-1}-v_{d-2})h_n$} (2.south east)
	
	(0.south west) edge[bend left=100, min distance=4.5cm] node[label, xshift=0 em]{$v_d h_n$} (5.south west)
	(1.south west) edge[bend left=75, min distance = 2.8 cm] node[label, xshift=0.2 em]{$v_{d-1} h_n$} (5.south west)
	(2.south west) edge[bend left=35] node[label, xshift=0.2 em]{$v_{d-2}h_n$} (5.south west)

	(0.south east) edge[bend right=45] node[label,xshift=6 em]{$(v_d-v_{d-1})h_n $} (1.south east);     
\end{tikzpicture}}\caption{\label{last}}\end{figure}
so that there are enough spacers $s_n$ that all elements of $D(I,N)-D(I,N)$ have a unique representation in $(h_i)$\\

Then it is clear that this transformation is $v$-positive type. To see that it is only $v$-positive type for $w\le_p v$, if it $w$-positive type, we must have an infinite sequence of $n$ with $w_in\in \{ M(n) - D(I,N)\}$, where $M(n)$ is fixed and in $D(I,N)$. Say $w_in$ is expanded with coefficients $A_i$ in $(h_i)$. Then note again that through the unique representation of $w_in$, we pick out the largest term $A_ih_{\ell_i}$ for each $1\le i\le d'$  and $B$, the largest term of $M(n)$ (such that it is no larger than the largest term of $w_i$, that is).  All $h_{\ell_i}$ must be equal, once we choose $n$ large enough compared to the terms of $w$. Then, as before, $\frac{w_i}{w_j}=\frac{A_i}{A_j}$ but furthermore  by our construction, $A_i, A_j \in \{ v_1-v_J, v_2-v_J, \ldots v_d - v_J\}$ where $v_J$ depends only on the value of $B$, and then the equation $w_i = w_1(A_i)/A_1$ actually implies that $w_i = w_1(v_k-v_J)/(v_j-v_J)$ and this gives the correct result, as the only term that varies with $w_i$ is $v_k$.

\end{proof}

\begin{corollary}
Let $\mathcal{V}=\left\lbrace v\right\rbrace $ be a finite family of finite vectors. Then there exists a rank-one transformation $T$ which is $v$ multiplicative-positive type, for all $v\in \mathcal{V}$, but which is not multiplicative-positive type for any finite vector $w$ where $w\not \leq_m v$ for all $v \in \mathcal{V}$. Similarly for positive type. 
\end{corollary}
\begin{proof}
We will use the same process as in the proof of the theorem above to produce a transformation which has the requested properties. If $\mathcal{V} = V_1, V_2,...V_D$ we define the transformation $T$ as follows: for each $m>1$ with largest prime power $p_j$ we divide $C_m$ into $d+1$ columns where $V_i\in \mathbb{Z}^d(V)$ so that $V_i = (V_{i,1}, V_{i,2}, \ldots, V_{i,d(V)})$ for $i= j\mod D$ and then we add spacers exactly as above except for the last column, where there are enough spacers added so that $h_{m+1} \geq H_M$ where $H_M=\sum_{k\leq m} \sum_{i\leq m}\sum_{1\leq j \leq d(V)} V_{i,j} h_k$. \\

Because for any prime there are an infinite number of powers $p^i$, for any $V_j \in \mathcal{V}$, the corresponding prime $p$ gives that $$\mu\left(I \cap T^{h_{p^i} V_{j,1}} I\right) \ldots \mu\left(I\cap T^{h_{p^i} V_{j,d(V_j)}}I\right)\geq \left(\frac{\mu(I)}{d(V_j)+1}\right)^{d(V_j)}$$  and so $T$ is $V_J$-positive type. \\ \\
However, using an argument similar to above
it can be seen that if $T$ is $w$ positive type, there must be a $v\in \mathcal{V}$ such that $w\leq_m v$.

Extension onto Theorem 7.3 is done exactly analogously. 
\end{proof}

This section essentially showed that except for the ``natural'' dependencies, different vectors $v,w\in (Z^+)^d$ behave independently for $v$-positive type.

\end{document}